\documentclass[12pt]{article}

\usepackage{latexsym,amssymb,upref,amsmath,amsthm, amsfonts,authblk}
\usepackage{amssymb,amsmath,amsthm, calc, graphicx}
\usepackage{epsfig}
\usepackage{breqn}
\usepackage{footnpag}
\usepackage{rotating}
\usepackage{amsfonts}
\usepackage{setspace}
\usepackage{fullpage}
\usepackage{enumitem}
\usepackage{bbold}
\usepackage{comment}
\usepackage{pgf,tikz}
\usepackage{mathrsfs}
\usetikzlibrary{arrows}
\usepackage{yhmath}
\usepackage{hyperref}
\usepackage{authblk}
\usepackage{comment}
\usepackage{enumitem}

\newtheorem{thm}{Theorem}%[section]
\newtheorem{definition}[thm]{Definition}
\newtheorem{claim}[thm]{Claim}
\newtheorem{lemma}[thm]{Lemma}

\newtheorem{observation}[thm]{Observation}

\newtheorem*{claimunnumbered}{Claim}

\newcommand{\thistheoremname}{}
\newtheorem*{genericthm*}{\thistheoremname}
\newenvironment{namedthm*}[1]
{\renewcommand{\thistheoremname}{#1}%
	\begin{genericthm*}}
	{\end{genericthm*}}

\newcommand\ex{\ensuremath{\mathrm{ex}}}

\newcommand{\B}{\mathcal{B}}

\newcommand{\Q}{\mathcal{Q}}

\newcommand{\cH}{\mathcal{H}}

\newcommand{\D}{\mathcal{D}}
\renewcommand{\P}{\mathcal{P}}

\newcommand{\codeg}{\mathrm{codeg}}
\newcommand{\abs}[1]{\left\lvert{#1}\right\rvert}

\title{$3$-uniform hypergraphs without a cycle of length five}

\author
{
Beka Ergemlidze
\thanks{Alfr\'ed R\'enyi Institute of Mathematics, Budapest.
		E-mail: \texttt{beka.ergemlidze@gmail.com}} \qquad 
Ervin Gy\H{o}ri
\thanks{Alfr\'ed R\'enyi Institute of Mathematics, Budapest.
		E-mail: \texttt{gyori.ervin@renyi.mta.hu}} 	\qquad
Abhishek Methuku \thanks{Department of Mathematics, \'Ecole Polytechnique F\'ed\'erale de Lausanne, Switzerland. E-mail: \texttt{abhishekmethuku@gmail.com}}
}

\begin{document}

\maketitle

\begin{abstract}
In this paper we show that the maximum number of hyperedges in a $3$-uniform hypergraph on $n$ vertices without a (Berge) cycle of length five is less than $(0.254 + o(1))n^{3/2}$, improving an estimate of Bollob\'as and Gy\H{o}ri. 

We obtain this result by showing that not many $3$-paths can start from certain subgraphs of the shadow.
\end{abstract}

\section{Introduction}

A hypergraph $H = (V, E)$ is a family $E$ of distinct subsets of a finite set $V$. The members of $E$ are called \emph{hyperedges} and the elements of $V$ are called \emph{vertices}. A hypergraph is called $r$-uniform is each member of $E$ has size $r$. A hypergraph $H = (V, E)$ is called \emph{linear} if every two hyperedges have at most one vertex in common. 

%A hypergraph is $\mathcal F$-free if it does not contain any member of $\mathcal F$ as a subhypergraph. A $2$-uniform hypergraph is simply called a graph. 
\vspace{2mm}

A Berge cycle of length $k \ge 2$, denoted Berge-$C_k$, is an alternating sequence of distinct vertices and distinct edges of the form $v_1,h_{1},v_2,h_{2},\ldots,v_k,h_{k}$ where $v_i,v_{i+1} \in h_{i}$ for each $i \in \{1,2,\ldots,k-1\}$ and $v_k,v_1 \in h_{k}$. (Note that if a hypergraph does not contain a Berge-$C_2$, then it is linear.) This definition of a hypergraph cycle is the classical definition due to Berge. More generally, if $F=(V(F),E(F))$ is a graph and $\Q=(V(\Q),E(\Q))$ is a hypergraph, then we say $\Q$ is \textit{Berge-F} if there is a bijection $\phi:E(F) \rightarrow E(\Q)$ such that $e \subseteq \phi(e)$ for all $e \in E(F)$. In other words, given a graph $F$ we can obtain a Berge-$F$ by replacing each edge of $F$ with a hyperedge that contains it. 

%This definition of a hypergraph cycle is the classical definition due to Berge. The notion of Berge cycles was recently generalized to arbitrary Berge graphs in \cite{gp1, GMT} and the linear Tur\'an number of (Berge) $K_{2,t}$ was studied in \cite{T2016} and \cite{GMM}. 
 
\vspace{2mm}

Given a family of graphs $\mathcal F$, we say that a hypergraph $\cH$ is \textit{Berge-$\mathcal F$-free} if for every $F \in \mathcal F$, the hypergraph $\cH$ does not contain a Berge-$F$ as a subhypergraph. The maximum possible number of hyperedges in a Berge-$\mathcal F$-free hypergraph on $n$ vertices is the \emph{Tur\'an number} of Berge-$\mathcal F$, and is denoted by $\ex_3(n, \mathcal F)$.  When $\mathcal F = \{F\}$ then we simply write $\ex_3(n, F)$ instead of $\ex_3(n, \{F\})$. 
%Given a family of $3$-uniform hypergraphs $\mathcal F$, let $\ex_3(n, \mathcal F)$ denote the maximum number of hyperedges of an $\mathcal F$-free $3$-uniform hypergraph on $n$ vertices and similarly, given a family of $3$-uniform linear hypergraphs $\mathcal F$, the \emph{linear Tur\'an number} of $\mathcal F$, denoted $\ex^{\lin}_3(n, \mathcal F)$, is the maximum number of hyperedges in an $\mathcal F$-free $3$-uniform linear hypergraph on $n$ vertices. When $\mathcal F = \{F\}$ then we simply write $\ex_3(n, F)$ instead of $\ex_3(n, \{F\})$. 

\vspace{2mm}

Determining $\ex_3(n, \{C_2, C_3\})$ is basically equivalent to the famous $(6,3)$-problem. This was settled by Ruzsa and Szemer\'edi in their classical paper \cite{Ruzsa_Szem}, showing that $n^{2-\frac{c}{\sqrt{\log n}}} < \ex_3(n, \{C_2, C_3\}) = o(n^2)$  for some constant $c > 0$. An important Tur\'an-type extremal result for Berge cycles is due to Lazebnik and Verstra\"ete \cite{Lazeb_Verstraete}, who studied the maximum number of hyperedges in an $r$-uniform hypergraph containing no Berge cycle of length less than five (i.e., girth five). They showed the following.

\begin{thm}[Lazebnik, Verstra\"ete \cite{Lazeb_Verstraete}]\label{lazver} We have $$\ex_3(n, \{C_2, C_3,C_4\})=\frac{1}{6} n^{3/2}+o(n^{3/2}).$$
\end{thm}

The systematic study of the Tur\'an number of Berge cycles started with the study of Berge triangles by Gy\H ori \cite{Gyori_triangle}, and continued with the study of Berge five cycles by Bollob\'as and Gy\H ori \cite{BGY2008} who showed the following.

\begin{thm}[Bollob\'as, Gy\H ori \cite{BGY2008}]
\label{BollobasGyorithm}
We have,
$$(1 + o(1)) \frac{n^{3/2}}{3\sqrt{3}} \le ex_3(n, C_5) \le \sqrt{2}n^{3/2} + 4.5n.$$
\end{thm}

The following construction of Bollob\'as and Gy\H ori proves the lower bound in Theorem \ref{BollobasGyorithm}.

\vspace{2mm}

\textit{Bollob\'as-Gy\H{o}ri Example.} Take a $C_4$-free bipartite graph $G_0$ with $n/3$ vertices in each part and $(1 + o(1))(n/3)^{3/2}$ edges. In one part, replace each vertex
$u$ of $G_0$ by a pair of two new vertices $u_1$ and $u_2$, and add the triple $u_1u_2v$ for each edge $uv$ of $G_0$. It is easy to check that the resulting hypergraph $H$ does not contain a Berge cycle of length $5$. Moreover, the number of hyperedges in $H$ is the same as the number of edges in $G_0$.

\vspace{2mm}

In this paper, we improve Theorem \ref{BollobasGyorithm} as follows.

\begin{thm}
We have,
	\label{mainbergec5}
	$$\ex_3(n, C_5) <  (1 + o(1)) 0.254  n^{3/2}.$$
\end{thm}

\vspace{2mm}
Roughly speaking, our main idea in proving the above theorem is to analyze the structure of a Berge-$C_5$-free hypergraph, and use this structure to efficiently bound the number of paths of length $3$ that start from certain dense subgraphs (e.g., triangle, $K_4$) of the $2$-shadow. This bound is then combined with the lower bound on the number of paths of length $3$ provided by the Blakley-Roy inequality  \cite{Blakley_Roy}. We prove Theorem \ref{mainbergec5} in Section \ref{maintheoremsection}.

\vspace{2mm}

Ergemlidze, Gy\H{o}ri and Methuku \cite{EGM}   considered the analogous question for linear hypergraphs and proved that $ex_3(n, \{C_2, C_5\}) = n^{3/2}/{3\sqrt{3}}+o(n^{3/2})$. Surprisingly, even though their lower bound is the same as the lower bound in Theorem \ref{BollobasGyorithm}, the linear hypergraph that they constructed in \cite{EGM} is very different from the hypergraph used in the Bollob\'as-Gy\H{o}ri example discussed above -- the latter is far from being linear. In \cite{EGM}, the authors also strengthened Theorem \ref{lazver} by showing that $\ex_3(n, \{C_2,C_3,C_4\}) \sim \ex_3(n, \{C_2,C_4\}).$ Recently, $\ex_3(n, C_4)$ was studied in \cite{C_4}.
See \cite{EM} for results on the maximum number of hyperedges in an $r$-uniform hypergraph of girth six. 

\vspace{2mm}

Gy\H{o}ri and Lemons \cite{Gyori_Lemons, Gyori_Lemons2} generalized Theorem \ref{BollobasGyorithm} to Berge cycles of any given length and proved bounds on $ex_r(n, C_{2k+1})$ and $ex_r(n, C_{2k})$. These bounds were improved by F\"uredi and \"Ozkahya \cite{Furedi_Ozkahya}, Jiang and Ma \cite{JiangMa}, Gerbner, Methuku and Vizer \cite{GMM}. Recently F\"uredi, Kostochka and Luo \cite{furedi2018avoiding} started the study of the maximum size of an $n$-vertex $r$-uniform hypergraph without any Berge cycle of length at least $k$. This study has been continued in \cite{furedi2018avoiding2, Gyori_Lemonsetal, KostochkaLuo, EGMSTZ}.

%Note that this upper bound has the same order of magnitude as the upper bound on the maximum possible number of edges in a $C_{2k}$-free graph (see the Even Cycle theorem of Bondy and Simonovits \cite{Bondy_Simonovits}). This shows the surprising fact that the maximum number of hyperedges in a Berge-$C_{2k+1}$-free hypergraph is significantly different from the maximum possible number of edges in a $C_{2k+1}$-free graph. Gyori and Lemons also showed that $ex_3(n,C_{2k}) \le O(k^2) \cdot ex(n, C_{2k})$. 

%improved this result by showing that $ex_3(n,C_{2k+1}) \le (9k^2 + 10k + 5)n^{1+1/k} + O(k^2n)$.  

%For $k \ge 2$, F\"uredi and \"Ozkahya \cite{Furedi_Ozkahya} showed $\ex^{\lin}_3(n, C_{2k+1}) \le 2kn^{1+1/k} + 9kn$. In fact it is shown in \cite{Gyori_Lemons, Furedi_Ozkahya} that $\ex_3(n, C_{2k+1}) \le O(n^{1+1/k})$. For the even case it is easy to show $\ex^{\lin}_3(n, C_{2k}) \le \ex(n, C_{2k}) = O(n^{1+1/k})$ by selecting a pair from each hyperedge of a $C_{2k}$-free $3$-uniform linear hypergraph. A (Berge) path of length $k$ is an alternating sequence of distinct vertices and distinct edges of the form $v_0, h_0, v_1,h_{1},v_2,h_{2},\ldots, v_{k-1}, h_{k-1},v_k$ where $v_i,v_{i+1} \in h_{i}$ for each $i \in \{0, 1,2,\ldots,k-1\}$.
\vspace{2mm}

General results for Berge-$F$-free hypergraphs have been obtained in \cite{gp1, GMT, GMP} and the Tur\'an numbers of Berge-$K_{2,t}$ and Berge cliques, among others, were studied in \cite{T2016, PTTW, GMM, Gyarfas, GMP}. 

%Below we concentrate on the linear Tur\'an numbers of $C_3$, $C_4$ and $C_5$.

% The following is our main result.

\subsection*{Notation}
	\label {notation}
We introduce some important notations and definitions used throughout the paper.  %When we say a path $v_0v_1\ldots v_k$ it means 
\begin{itemize}
\itemsep0em 
    \item Length of a path is the number of edges in the path. We usually denote a path $v_0, v_1, \ldots, v_k$, simply as $v_0 v_1 \ldots v_k$.

    \item For convenience, an edge $\{a,b\}$ of a graph or a pair of vertices $a,b$ is referred to as $ab$. A hyperedge $\{a,b,c\}$ is written simply as $abc$. 
    
    \item For a hypergraph $H$ (or a graph $G$), for convenience, we sometimes use $H$ (or $G$) to denote the edge set of the hypergraph $H$ (or $G$ respectively). Thus the number of edges in $H$ is $\abs{H}$.
    
    \item Given a graph $G$ and a subset of its vertices $S$, let the subgraph of $G$ induced by $S$ be denoted by $G[S]$. 
    
    \item  For a hypergraph $H$, let $\partial H = \{ab \mid ab \subset e \in E(H)\}$ denote its \emph{2-shadow} graph. 
    
    \item For a hypergraph $H$, the \textit{neighborhood} of $v$ in $H$ is defined as 
$$N(v) = \{x \in V(H) \setminus \{v \}  \mid v,x \in h \text{ for some } h \in E(H)\}.$$

    \item For a hypergraph $H$ and a pair of vertices $u, v \in V(H)$, let $\codeg(v,u)$ denote the number of hyperedges of $H$ containing the pair $\{u,v\}$.
\end{itemize}

%(Notice that the basic cycle of a berge cycle $C_k$ is a cycle in the graph $\partial{C_k}$.) If $H$ is linear, then $\abs{E(\partial H)} = 3 \abs{E(H)}$. For a hypergraph $H$ and $v \in V(H)$, we denote the degree of $v$ in $H$ by $d(v)$. We write $d^H(v)$ instead of $d(v)$ when it is important to emphasize the underlying hypergraph. 
% and 
 %$$N^H_2(v) = \{x \in V(H) \setminus (N^H_1(v) \cup \{v\}) \mid \exists h \in E(H) \text{ such that } x \in h \text{ and } h \cap N^H_1(v) \not = \emptyset\}$$
 % respectively.
%For convenience we use $N(v)$ and $N_2(v)$.
%\begin{definition}
%A path in $\partial(H)$ is called \emph{Bad} if it has a \emph{Bad} subpath, otherwise we call it a \emph{Good} path.
%\end{definition}

\newpage

\section{Proof of Theorem \ref{mainbergec5}}
\label{maintheoremsection}

Let $H$ be a hypergraph on $n$ vertices without a Berge $5$-cycle and let $G=\partial H$ be the $2$-shadow of $H$. First we introduce some definitions.

%\begin{definition}
%Let $abc$ be a hyperedge of a hypergraph $H$. We say the pair $ab$ is \emph{thin} if there is only one (in this case $abc$) hyperedge containing it, otherwise we call $ab$ a \emph{fat} pair.
%We say a hyperedge $abc$ is \emph{thin} if at least two of the pairs $ab,bc,ac$ are thin, otherwise we say that the hyperedge is \emph{fat}.
%\end{definition}
\begin{definition}
A pair $xy \in \partial H$ is called \emph{thin} if $\codeg(xy) = 1$, otherwise it is called \emph{fat}. 

We say a hyperedge $abc \in H$ is \emph{thin} if at least two of the pairs $ab,bc,ac$ are thin.
\end{definition}

\begin{definition}
We say a set of hyperedges (or a hypergraph) is tightly-connected if it can be obtained by starting with a hyperedge and adding hyperedges one by one, such that every added hyperedge intersects with one of the previous hyperedges in $2$ vertices.
\end{definition}

\begin{definition}
A \emph{block} in  $H$ is a maximal set of tightly-connected hyperedges.
%\{$h_1$,$h_2$ ... $h_k$\} of hyperedges of $H$,  such that for each $2\leq i \leq k$ there exists $1\leq j < i$ with $|h_i \cap h_j|=2$
\end{definition}

\begin{definition}
For a block $B$,
%be a \epmh{block}, a hyperedge $abc\in B$ is called a \emph{Thin} if at least two of the pairs \{$ab, bc, ac$\} are Thin.
a maximal subhypergraph of $B$ without containing thin hyperedges is called the \emph{core} of the block.
\end{definition}

Let $K_4^3$ denote the complete $3$-uniform hypergraph on $4$ vertices. A crown of size $k$ is a set of $k\geq 1$ hyperedges of the form $abc_1,abc_2,\ldots ,abc_k$. Below we define $2$ specific hypergraphs:
\begin{itemize}
%\item Let $F_1$ be a hypergraph consisting of one single hyperedge.
%\item Let $F_2$ be a hypergraph consisting of two hyperedges sharing two vertices.
\item Let $F_1$ be a hypergraph consisting of exactly $3$ hyperedges on $4$ vertices (i.e., $K_4^{3}$ minus an edge).
%\item For distinct vertices $a,b,c,d$ and $o$, let $F_4$ be a hypergraph consisting of hyperedges $abc$, $bcd$, $oab$ and $ocd$.
\item For distinct vertices $a,b,c,d$ and $o$, let $F_2$ be the hypergraph consisting of hyperedges $oab, obc, ocd$ and $oda$.
\end{itemize}

%Notice that $F_2$ is maximal tighly-connected set, otherwise we would easily find a Berge $5$-cycle.

%For a given $3$-uniform hypergraph $F$ For simplicity, we say a subhypergraph of $H$ is 
%We say a hypergraph is $F_1$ or $F_2$ if it is isomorphic to $F_1$ or $F_2$ respectively.

%\begin{observation}
%\label{maximalblock}
%If $H$ has a subhypergraph isomorphic to $F_4$ or $F_2$, then it is a block in $H$.
%\end{observation}

\begin{lemma}
\label{blockstructure}
Let $B$ be a block of $H$, and let $\B$ be a core of $B$. Then $\B$ is either $\emptyset, K_4^3, F_1, F_2$ or a crown of size $k$ for some $k\geq 1$.
\end{lemma}
\begin{proof}
If $\B= \emptyset$, we are done, so let us assume $\B \not = \emptyset$. Since $B$ is tightly-connected and it can be obtained by adding thin hyperedges to $\B$, it is easy to see that  $\B$ is also tightly-connected. Thus if $\B$ has at most two hyperedges, then it is a crown of size $1$ or $2$ and we are done. Therefore, in the rest of the proof we will assume that $\B$ contains at least $3$ hyperedges.

If $\B$ contains at most $4$ vertices then it is easy to see that $\B$ is either $K_4^3$ or $F_1$. So assume that $\B$ has at least $5$ vertices (and at least $3$ hyperedges). Since $\B$ is not a crown, there exists a tight path of length $3$, say $abc,bcd,cde$. Since $abc$ is in the core, one of the pairs $ab$ or $ac$ is fat, so there exists a hyperedge $h\not =abc$ containing either $ab$ or $ac$. Similarly there exists a hyperedge $f\not = cde$ and $f$ contains $ed$ or $ec$. If $h=f$ then $\B \supseteq F_2$. However, it is easy to see that $F_2$ cannot be extended to a larger tightly-connected set of hyperedges without creating a Berge $5$-cycle, so in this case $\B = F_2$. If $h\not =f$ then the hyperedges $h,abc,bcd,cde,f$ create a Berge $5$-cycle in $H$, a contradiction.
%If $h=f$ then $h=f=ace$ so $\B$ contains $F_2=\{abc,bcd,cde,ace\}$, and by observation \ref{maximalblock} $B=\B=\F_2$.
This completes the proof of the lemma.
\end{proof}

%%%%%%%%%%%%%%%%%%%%%%%%%%%%%%%%%%%%%%%%%%%%%%%%%%%%%%%%%%%%%%%%%%%%%%%%%%%
\begin{observation}
\label{aboutcore}
%Notice that every non-empty core is tightly-connected.
%If $\B\not =\emptyset$ then $B$ is a crown.
%If $abc\in B$ is a hyperedge and $ab, bc$ are thin pairs then either $B$ is a crown or $ac\in \partial \B$.
Let $B$ be a block of $H$ and let $\B$ be the core of $B$. If $\B = \emptyset$ then the block $B$ is a crown, and if $\B \not = \emptyset$ then every fat pair of $B$ is contained in $\partial \B$.

\end{observation}

%%%%%%%%%%%%%%%%%%%%%%%%%%%%%%%%%%%%%%%%%%%%%%%%%%%%%%%%%%%%%%%%%%%%%%%%%%%
\vspace{3mm}

\begin{minipage}{\textwidth-3mm}
\textbf{Edge Decomposition of $G=\partial H$.} 
We define a decomposition $\mathcal D$ of the edges of $G$ into paths of length 2, triangles and $K_4$'s such as follows:

Let $B$ be a block of $H$ and $\B$ be its core. 

\vspace{1mm}

If $\B=\emptyset$, then $B$ is a crown-block $\{abc_1, abc_2, \ldots, abc_k\}$  (for some $k \ge 1$); we partition $\partial B$ into the triangle $abc_1$ and paths $ac_ib$ where $2 \le i \le k$. 

\vspace{1mm}

If $\B \not=\emptyset$, then our plan is to first partition $\partial B\setminus \partial \B$.
If $abc\in B\setminus \B$, then $abc$ is a thin hyperedge, so it contains at least $2$ thin pairs, say $ab$ and $bc$. We claim that the pair $ac$ is in $\partial{\B}$. Indeed, $ac$ has to be a fat pair, otherwise the block $B$ consists of only one hyperedge $abc$, so $\B =\emptyset$ contradicting the assumption. So by Observation \ref{aboutcore}, $ac$ has to be a pair in $\partial{\B}$. For every $abc\in B\setminus \B$ such that $ab$ and $bc$ are thin pairs, add the $2$-path $abc$ to the edge decomposition $\mathcal D$. This partitions all the edges in $\partial B\setminus \partial \B$ into paths of length $2$. So all we have left is to partition the edges of $\partial \B$.
\begin{itemize}
\itemsep0em 
    \item If $\B$ is a crown $\{abc_1, abc_2, \ldots, abc_k\}$ for some $k\geq 1$, then we partition $\partial B$ into the triangle $abc_1$ and paths $ac_ib$ where $2 \le i \le k$.
    \item If $\B=F_1=\{abc,bcd,acd\}$ then we partition $\partial \B$ into $2$-paths $abc,bdc$ and $cad$.
    \item If $\B=F_2=\{oab,obc,ocd,oda\}$ then we partition $\partial \B$ into $2$-paths $abo, bco, cdo$ and $dao$.
    \item Finally, if $\B=K_4^3=\{abc,abd,acd,bcd\}$ then we partition $\partial \B$ as $K_4$, i.e., we add $\partial \B=K_4$ as an element of $\D$.
\end{itemize}
Clearly, by Lemma \ref{blockstructure} we have no other cases left. Thus all of the edges of the graph $G$ are partitioned into paths of length 2, triangles and $K_4$'s.
\end{minipage}

\vspace{3mm}

\begin{observation}
\label{2pathfatedge}

\begin{itemize}

\item [ ]
\item [(a)] If $D$ is a triangle that belongs to $\D$, then there is a hyperedge $h\in H$ such that $D= \partial h$.

\item [(b)] If $abc$ is a $2$-path that belongs to $\D$, then $abc \in H$. Moreover $ac$ is a fat pair.

\item [(c)]If $D$ is a $K_4$ that belongs to $\D$, then there exists $F=K_4^3 \subseteq H$ such that $D = \partial F$.

%\item [(d)]Let $F=K_4^3\subset H$, then $\partial F = K_4 \in \D$.
%\item [(e)]Let $h = abc$ be a hyperedge of $H$ which is not contained in $K_4^3$. Then either one of the $2$-paths $abc, bca, cab$ or the triangle $abc$ is in $\D$.
\end{itemize}

\end{observation}

Let $\alpha_1 \abs{G}$ and $\alpha_2 \abs{G}$ be the number of edges of $G$ that are contained in triangles and $2$-paths of the edge-decomposition $\mathcal D$ of $G$, respectively. So $(1-\alpha_1-\alpha_2) \abs{G}$ edges of $G$ belong to the $K_4$'s in $\mathcal D$.

\begin{claim}
\label{alpha12}
We have,
$$\abs{H}=\left(\frac{\alpha_1}{3}+\frac{\alpha_2}{2}+\frac{2(1-\alpha_1-\alpha_2)}{3}\right)\abs{G}.$$
\end{claim}
\begin{proof}
Let $B$ be a block with the core $\B$. Recall that for each hyperedge $h \in B \setminus \B$, we have added exactly one $2$-path or a triangle to $\D$.
%(these $2$-paths or triangles are inside $\partial h$). 

Moreover, because of the way we partitioned $\partial \B$, it is easy to check that in all of the cases except when $\B = K_4^3$, the number of hyperedges of $\B$ is the same as the number of elements of $\mathcal D$ that $\partial \B$ is partitioned into; these elements being $2$-paths and triangles. On the other hand, if $\B = K_4^3$, then the number of hyperedges of $\B$ is $4$ but we added only one element to $\D$ (namely $K_4$). 

This shows that the number of hyperedges of $H$ is equal to the number of elements of $\D$ that are $2$-paths or triangles plus the number of hyperedges which are in copies of $K_4^3$ in $H$, i.e., $4$ times the number of $K_4$'s in $\D$. Since $\alpha_1 \abs{G}$ edges of $G$ are in $2$-paths, the number of elements of $\D$ that are $2$-paths is $\alpha_1 \abs{G}/2$. Similarly, the number of elements of $\D$ that are triangles is $\alpha_2 \abs{G}/3$, and the number of $K_4$'s in $\D$ is $(1-\alpha_1-\alpha_2) \abs{G}/6$. Combining this with the discussion above finishes the proof of the claim.
%First, we claim that for each $h\in H$ such that $h$ does not belong to any $K_4^3$ of $H$, there exists $D\in \D$ such that $\partial D\in h$. Let $h$ be such hyperedge. let $B$ be a block such that $h\in B$ and let $\B$ be a core of $B$. 
%If $h\in B\setminus \B$ then $h$ is a thin hyperedge
%If $\B=\emptyset$ then $B$ is a crown and according to our decomposition either $\partial h\in D$ or some $2$-path of $\partial h$ is in $\D$. 
%If $\B\not=\emptyset$ and $h\in B\setminus \B$ then some $2$-path of $\partial h$ is an element of $\D$.
%If $h\in \B$ ITS FUCKED UP WHIS MAKES NO SENSE
\end{proof}

%\begin{observation}
%\label{2pathfatedge}
%For every $2$-path $P=abc$ in $\D$,  $ac$ is a fat pair.
%\end{observation}

%\begin{proof}
%Clearly there is a hyperedge $h=abc$. If $ac$ is thin then it can not be Let $B$ be the block whose shadow contains the path $abc$, and let $\B$ be the core of it. If $abc\in B\setminus B$
%\end{proof}

%If $\B=\F_1=abc$ then we partition $\partial \B=abc$ into the triangle $abc$. 

%If $\B=\F_2=\{abc,bcd\}$ then we partition $\partial \B$ into $abc$ $2$-path and a triangle $bcd$.

%Each edge of $G$ belongs to a triangle which is a hyeredge in $H$, and all the hyperedges of $H$ are partitioned into blocks, so it follows that the edges of $G$ are partitioned into blocks as well. Moreover, by Claim \ref{Blocks}, edges of $G$ can be decomposed into crown-blocks and $K_4$-blocks. We further partition the edges of each crown-block $\{abc_1, abc_2, \ldots, abc_k\}$  (for some $k \ge 1$) into the triangle $abc_1$ and paths $ac_ib$ where $2 \le i \le k$. This gives the desired decomposition $\mathcal D$ of $E(G)$.

The link of a vertex $v$ is the graph consisting of the edges $\{uw \mid uvw \in H \}$ and is denoted by $L_v$.

\begin{claim}
\label{linksize}
$\abs{L_v}\leq 2\abs{N(v)}$. %Morover, if $S\subseteq N(v)$ then $\abs{L_v[S]} \leq 2\abs{S}$.
\end{claim} 
\begin{proof}
First let us notice that there is no path of length $5$ in $L_v$. Indeed, otherwise, there exist vertices $v_0,v_1, \ldots, v_5$ such that $vv_{i-1}v_i\in H$ for each $1\leq i \leq 5$ which means there is a Berge $5$-cycle in $H$ formed by the hyperedges containing the pairs $vv_1, v_1v_2, v_2v_3, v_3v_4, v_4v$, a contradiction. So by the Erd\H os-Gallai theorem $\abs{L_v}\leq \frac{5-1}2\abs {N(v)}$, proving the claim.
%$|L(v)|\leq 2\abs{N(v)}$. First let us notice that there is no path of length $5$ in $L(v)$, if theres is, then there exists vertices $v_0,v_1 \ldots v_5$ such that $vv_{i-1}v_i\in H$ for each $1\leq i \leq 5$ which means there is a Berge $5$-cycle on base vetices $v,v_1,v_2,v_3,v_4$. So by Erd\H os-Gallai theorem $L(v)\leq \frac{5-1}2\abs {N(v)}$.
\end{proof}

\begin{lemma}
\label{neighborhood edges}
Let $v\in V(H)$ be an arbitrary vertex, then the number of edges in $G[N(v)]$ is less than $8\abs{N(v)}$.
\end{lemma}
\begin{proof}
 %First we prove the following claim.
%\begin{claim}
%\label{neighborhood paths}
Let $G_v$ be a subgraph of $G$ on a vertex set $N(v)$, such that $xy \in G_v$ if and only if there exists a vertex $z \not=v$ such that $xyz\in H$. Then each edge of $G[N(v)]$ belongs to either $L_v$ or $G_v$, so $\abs{G[N(v)]}\leq \abs{L_v}+ \abs{G_v}$. Combining this with Claim \ref{linksize}, we get $\abs{G[N(v)]}\leq \abs{G_v}+2\abs{N(v)}$. So it suffices to prove that $\abs{G_v}<6\abs{N(v)}$.
%Let $G_v$ be a subgraph of $G$ on a vertex set $N(v)$, such that $xy \in G_v \Leftrightarrow \exists z \not=v, xyz\in H$. Then $\abs{G_v}<6\abs{N(v)}$.
%\end{claim}

%begin{proof}
First we will prove that there is no path of length $12$ in $G_v$.
Let us assume by contradiction that $P=v_0,v_1, \ldots, v_{12}$ is a path in $G_v$.
Since for each pair of vertices $v_i,v_{i+1}$, there is a hyperedge $v_iv_{i+1}x$ in $H$ where $x\not= v$, we can conclude that there is a subsequence $u_0,u_1, \ldots, u_6$ of $v_0, v_1, \ldots, v_{12}$ and a sequence of distinct hyperedges $h_1, h_2, \ldots, h_6$, such that $u_{i-1}u_i \subset h_i$ and $v\notin h_i$ for each $1\leq i \leq 6$. Since $u_0,u_3,u_6\in N(v)$ there exist hyperedges $f_1,f_2,f_3\in H$ such that $vu_0 \subset  f_1$, $vu_3 \subset f_2$ and $vu_6 \subset f_3$. Clearly, either $f_1\not =f_2$ or $f_2\not =f_3$. In the first case the hyperedges $f_1,h_1,h_2,h_3,f_2$, and in the second case the hyperedges $f_2,h_4,h_5,h_6,f_3$ form a Berge $5$-cycle in $H$, a contradiction.

Therefore, there is no path of length $12$ in $G_v$, so by the Erd\H os-Gallai theorem, the number of edges in $G_v$ is at most $\frac{12-1}2\abs{N(v)}<6\abs{N(v)}$, as required.
%\end{proof}
%Since each edge of $G[S]$ belongs to either $L_v[S]$ or $G_v[S]$, by Claim \ref{linksize} and Claim \ref{neighborhood paths}, proving that $\abs{G[S]}<8\abs{S}$.
\end{proof}

\subsection{Relating the hypergraph degree to the degree in the shadow}
For a vertex $v\in V(H)=V(G)$, let $d(v)$ denote the degree of $v$ in $H$ and let $d_G(v)$ denote the degree of $v$ in $G$ (i.e., $d_G(v)$ is the degree in the shadow). 

Clearly $d_G(v)\leq 2d(v)$. Moreover, $d(v)=\abs{L_v}$ and $d_G(v)=\abs{N(v)}$. So by Claim \ref{linksize}, we have
\begin{equation}
\label{degreeconnection}
    \frac{d_G(v)}{2}\leq d(v) \leq 2 d_G(v).
\end{equation}

Let $\overline d$ and $\overline d_G$ be the average degrees of $H$ and $G$ respectively.

Suppose there is a vertex $v$ of $H$, such that $d(v) < \overline d/3$. Then we may delete $v$ and all the edges incident to $v$ from $H$ to obtain a graph $H'$ whose average degree is more than $3(n\overline d/3-\overline d/3)/(n-1) =\overline  d$. Then it is easy to see that if the theorem holds for $H'$, then it holds for $H$ as well. Repeating this procedure, we may assume that for every vertex $v$ of $H$, $d(v) \ge \overline d/3$. Therefore, by \eqref{degreeconnection}, we may assume that the degree of every vertex of $G$ is at least $\overline d/6$.

\subsection{Counting paths of length $3$}

\begin{definition}
A $2$-path in $\partial H$ is called \emph{bad} if both of its edges are contained in a triangle of $\partial H$, otherwise it is called \emph{good}.
\end{definition}

%$\binom{d_G(v)}{2}<\frac {d_G^2(v)}2$ $2$ paths and 

\begin{lemma}
\label{weirdlemma}
For any vertex $v\in V(G)$ and a set $M\subseteq N(v)$, let $\P$ be the set of the good $2$-paths $vxy$ such that $x\in M$. Let $M'=\{y\mid vxy\in \P\}$ then $\abs{\P}< 2\abs{M'}+48d_G(v)$.
\end{lemma}
\begin{proof}
Let $B_{\P}=\{xy \mid x \in M, y \in M', xy \in G \}$ be a bipartite graph, clearly $\abs{B_{\P}}=\abs{\P}$. Let $E=\{xyz\in H\mid x,y\in N(v), \codeg(x,y)\leq 2\}$. By Lemma \ref{neighborhood edges}, $\abs{E}\leq2\cdot 8\abs{N(v)}$ so the number of edges of $2$-shadow of $E$ is $\abs{\partial E}\leq 48\abs{N(v)}$. Let $B=\{xy\in B_{\P}\mid \exists z\in V(H), xyz\in H\setminus E \}$. Then clearly,
\begin{equation}
\label{boundonb}
    \abs{B}\geq \abs{B_{\P}}-\abs{{\partial E}}\geq \abs{\P}-48\abs{N(v)}=\abs{\P}-48d_G(v).
\end{equation}
Let $d_B(x)$ denote the degree of a vertex $x$ in the graph $B$.

\begin{claim}
\label{bipartitegraph}
For every $y\in M'$ such that $d_B(y)=k\geq 3$, there exists a set of $k-2$ vertices $S_y \subseteq M'$ such that  $\forall w \in S_y$ we have $d_B(w)=1$. Moreover, $S_y\cap S_z=\emptyset$ for any $y\not=z \in M'$ (with $d_B(y), d_B(z) \geq 3$).
\end{claim}

\begin{proof}
Let $yx_1,yx_2, \ldots ,yx_k\in B$ be the edges of $B$ incident to $y$. For each $1\leq j \leq k$ let $f_j\in H$ be a hyperedge such that $vx_j\subset f_j$. For each $yx_i\in B$ clearly there is a hyperedge $yx_iw_i\in H\setminus E$. 

We claim that for each $1\le i \le k$, $w_i\in M'$. It is easy to see that $w_i\in N(v)$ or $w_i\in M'$ (because $vx_iw_i$ is a $2$-path in $G$).  Assume for a contradiction that $w_i \in N(v)$, then since $yx_iw_i\notin E$ we have, $\codeg (x_i,w_i)\geq 3$. Let $f\in H$ be a hyperedge such that $vw_i\subset f$. Now take $j\not=i$ such that $x_j\not= w_i$. If $f_j\not =f$ then since $\codeg (x_i,w_i)\geq 3$ there exists a hyperedge $h\supset x_iw_i$ such that $h\not=f$ and $h\not=x_iw_iy$, then the hyperedges $f,h,x_iw_iy,yx_jw_j,f_j$ form a Berge $5$-cycle. So $f_j=f$, therefore $f_j\not=f_i$. Similarly in this case, there exists a hyperedge $h\supset x_iw_i$ such that $h\not=f_i$ and $h\not=x_iw_iy$, therefore the hyperedges $f_i,h,x_iw_iy,yx_jw_j,f_j$ form a Berge $5$-cycle, a contradiction. So we proved that $w_i\in M'$ for each $1\leq i \leq k$.

%$, w_ix_i\subset h, x_iy\subset h'$.   There exists a hyperedge $f'\in H$ such that $yx_j\subset f'$ (Notice that $f'\not=f_j$ because $y \notin N(v)$). Then if $h''\not= f_j$ the hyperedges $f_j, f', yx_iw_i,  h,  h''$ form a Berge cycle with base vertices $v,x_j,y,x_i,w_i$ and if $h'\not= f_j$ the hyperedges  $f_j, f',yx_iw_i,h, h'$ forms a Berge $5$-cycle with the base vertices $v,x_j,y,w_i,x_i $, a contradiction, i.e $w_i\in M'$ for each $1\le i \le k$. 

%Now let us prove that $d_B(w_i)=1$ for all but one $1\leq i\leq k$.
\begin{claimunnumbered}
For all but at most $2$ of the $w_i$'s (where $1\leq i\leq k$), we have $d_B(w_i)=1$.
\end{claimunnumbered}

\begin{proof} If $d_B(w_i)=1$ for all $1\leq i \leq k$ then we are done, so we may assume that there is $1\leq i \leq k$ such that $d_B(w_i)\not=1$. 

For each $1\leq i \leq k$, $w_i\in M'$ and $x_iw_i\in \partial (H\setminus E)$ (because $x_iw_iy \in H\setminus E$), so it is clear that $d_B(w_i)\geq 1$. So $d_B(w_i)>1$. Then there is a vertex $x\in M\setminus \{x_i\}$ such that $w_ix\in B$. Let $f,h\in H$ be hyperedges with $w_ix\in h$ and $xv\in f$. 
If there are $j,l\in \{1,2,\ldots ,k\}\setminus \{i\}$ such that $x,x_j$ and $x_l$ are all different from each other, then clearly, either $f\not=f_j$ or $f\not = f_l$, so without loss of generality we may assume $f\not=f_j$. Then the hyperedges $f,h,w_ix_iy,yw_jx_j,f_j$ create a Berge cycle of length $5$, a contradiction. 
So there are no $j,l\in \{1,2,\ldots ,k\}\setminus \{i\}$ such that $x,x_j$ and $x_l$ are all different from each other. Clearly this is only possible when $k<4$ and there is a $j\in \{1,2,3\}\setminus \{i\}$ such that $x=x_j$. Let $l\in \{1,2,3\}\setminus \{i,j\}$. If $f_j\not =f_l$ then the hyperedges $f_j,h,w_ix_iy,yw_lx_l,f_l$ form a Berge $5$-cycle. Therefore $f_j=f_l$. So we proved that $d_B(w_i)\not=1$ implies that $k=3$ and for $\{j,l\}=\{1,2,3\}\setminus \{i\}$, we have $f_j=f_l$. So if $d_B(w_i)\not =1$ and $d_B(w_j)\not =1$ we have $f_j=f_l$ and $f_i=f_l$, which is impossible. So $d_B(w_j)=1$. So we proved that if for any $1\leq i \leq k$, $d_B(w_i)\not=1$ then $k=3$ and all but at most $2$ of the vertices in $\{w_1,w_2,w_3\}$ have degree $1$ in the graph $B$, as desired.
%Now let us consider the vertex $w_j$, we claim that $d(w_i)=1$. Let us assume for a contradiction that there is a vertex $x'\in M$ such that $w_jx'\in B$ and $x'\not=x_i$
%If there are $j,l$ such that $x,x_j$ and $x_l$ are all different from each other, then either the hyperedges $f,h,y$
%$x=x_j$ for some $j\not =i$, let $l\in \{1,2,\ldots ,k\}\setminus \{i,j\}$, the one of the sets of hyperedges $yx_iw_i,f,f_i,f_l,x_lw_ly$ and $yx_jw_j,f_j,f_l,x_lw_ly$ form a Berge $5$-cycle. So we have $j,l\in\{1,2,\ldots ,k\}\setminus \{i\}$ such that $x_l,x_j$ and $x$ are all different. Then either $yx_iw_i,f,f_i,f_l,x_lw_ly$ or $yx_iw_i,f,f_i,f_j,x_jw_jy$ is a Berge $5$-cycle.
%So we proved that for all but at most $2$ of $w_i$ where $1\leq i\leq k$, $d_B(w_i)=1$. 
\end{proof}
We claim that for any $i\not =j$ where $d_B(w_i)=d_B(w_j)=1$ we have $w_i\not=w_j$. Indeed, if there exists $i\not =j$ such that $w_i=w_j$ then $w_ix_j$ and $w_ix_i$ are both adjacent to $w_i$ in the graph $B$ which contradicts to $d_B(w_i)=1$. So using the above claim, we conclude that the set $\{w_1,w_2,\ldots,w_k\}$ contains at least $k-2$ distinct elements with each having degree one in the graph $B$, so we can set $S_y$ to be the set of these $k-2$ elements. (Then of course $ \forall w_i\in S_y$ we have $d_B(w_i)=1$.) 

Now we have to prove that for each $z\not=y$ we have $S_y\cap S_z=\emptyset$. Assume by contradiction that $w_i\in S_z\cap S_y$ for some $z\not =y$. That is, there is some hyperedge $uw_iz\in H\setminus E$ where $u\in M$, moreover $u=x_i$ otherwise $d_B(w_i)>1$. So we have a hyperedge $x_iw_iz\in H\setminus E$ for some $z\in M'\setminus \{y\}$.
%Notice that the hyperedges $yx_iw_i,x_iw_iz$ and $f_i$ contribute to a Berdge $3$-path from $y$ to $v$ with the base path $yw_ix_iv$.
Let $j,l\in \{1,2,\ldots , k\}\setminus \{i\}$ such that $j\not=l$. Recall that $x_jv\subset f_j$ and  $x_lv\subset f_l$. Clearly either $f_j\not = f_i$ or $f_l\not= f_i$ so without loss of generality we can assume $f_j\not =f_i$. Then it is easy to see that the hyperedges $f_j,x_jw_jy,yx_iw_i,w_izx_i,f_i$ are all different and they create a Berge $5$-cycle ($x_jw_jy\not =yx_iw_i$ because $x_j\not =w_i$).
\end{proof}
%So let $M'\subset N_2(v)\mid x\in M' \Leftrightarrow d_B(x)>0$. 
For each $x\in M'$ with $d_B(x)=k\geq 3$, let $S_x$ be defined as in Claim \ref{bipartitegraph}. Then the average of the degrees of the vertices in $S_x\cup \{x\}$ in $B$ is $(k+\abs{S_x})/(k-1)= (2k-2)(k-1)=2$. Since the sets $S_x \cup x$ (with $x\in M'$, $d_B(x)\geq 3$) are disjoint, we can conclude that average degree of the set $M'$ is at most $2$. Therefore $2\abs{M'}\geq \abs{B}$. So by \eqref{boundonb} we have $2\abs{M'}\geq \abs{B}> \abs{\P}-48d_G(V)$, which completes the proof of the lemma. 
%So if $c>120$ then $d< \frac{\sqrt {n}}{\sqrt{3}}$ and we are done. So from now on we will assume that for every $x\in V(H)$ $d(x)\leq 120d$.
\end{proof}

\begin{claim}
\label{maxdegree}
We may assume that the maximum degree in the graph $G$ is less than $160\sqrt n$ when $n$ is large enough.
\end{claim}
\begin{proof}

Let $v$ be an arbitrary vertex with $d_G(v) = C\overline d$ for some constant $C>0$. Let $\P$ be the set of the good $2$-paths starting from the vertex $v$. 
%$vxy$ such that $x\in M$ (in this case all good $2$-paths in $G$ starting from $v$), let $M'=\{y\mid vxy\in \P\}$. 
Then applying Lemma \ref{weirdlemma} with $M = N(v)$ and $M'=\{y\mid vxy\in \P\}$, we have $\abs{\P}< 2\abs{M'}+48 d_G(v)<2n+48\cdot C\overline d$.
Since the minimum degree in $G$ is at least $\overline d/6$, the number of (ordered) $2$-paths starting from $v$ is at least $d(v)\cdot(\overline d/6-1)=C\overline d\cdot (\overline d/6-1)$. Notice that the number of (ordered) bad $2$-paths starting at $v$ is the number of $2$-paths $vxy$ such that $x,y\in N(v)$. So by Lemma \ref{neighborhood edges}, this is at most $2\cdot 8 \abs{N(v)}=16C\overline d$, so the number of good $2$-paths is at least $C\overline d\cdot (\overline d/6-17)$. So $\abs{\P} \ge C\overline d\cdot (\overline d/6-17)$. 
Thus we have $$C \overline d\cdot (\overline d/6-17)\leq \abs{\P}<2n+48C\overline d.$$ So $C\overline d(\overline d/6-65)<2n$. Therefore, $6C(\overline d/6-65)^2<2n$, i.e., $\overline d<6\sqrt{n/3C}+390$,  so $\abs{H}=n\overline d/3\leq 2n\sqrt{n/3C}+130n$. 
If $C\geq 36$ we get that $\abs {H}\leq \frac{n^{3/2}}{3\sqrt{3}}+130n = \frac{n^{3/2}}{3\sqrt{3}}+O(n)$, proving Theorem \ref{mainbergec5}. So we may assume $C<36$.

Theorem \ref{BollobasGyorithm} implies that 
\begin{equation}
\label{gyoribollobasberge}
    \abs{H}=n\overline d/3 \le \sqrt{2}n^{3/2}+4.5n,
\end{equation} so $\overline d \le 3\sqrt 2\sqrt{n}+13.5$. So combining this with the fact that $C<36$, we have $d_G(v) = C\overline d <  108 \sqrt{2} \sqrt{n}+486<160\sqrt n$ for large enough $n$.
\end{proof}

Combining Lemma \ref{weirdlemma} and Claim \ref{maxdegree}, we obtain the following.

\begin{lemma}
\label{weirdcor}
For any vertex $v\in V(G)$ and a set $M\subseteq N(v)$, let $\P$ be the set of good $2$-paths $vxy$ such that $x\in M$. Let $M'=\{y\mid vxy\in \P\}$ then $\abs{\P}< 2\abs{M'}+7680 \sqrt n$ when $n$ is large enough.
\end{lemma}

\begin{definition}
\label{3-walks}
%A $3$-path ($3$-walk) $x_0,x_1,x_2,x_3$ is called \emph{good} if both $2$-paths ($2$-walks) $x_0,x_1,x_2$ and $x_1,x_2,x_3$ are good $2$-paths, otherwise the $3$-path ($3$-walk) is called \emph{bad}.
A $3$-path $x_0,x_1,x_2,x_3$ is called \emph{good} if both $2$-paths $x_0,x_1,x_2$ and $x_1,x_2,x_3$ are good $2$-paths.
\end{definition}

\begin{claim}
\label{countingpaths}
The number of (ordered) good $3$-paths in $G$ is at least $n\overline d_G^3 - C_0n^{3/2}\overline d_G$ for some constant $C_0 >0$ (for large enough $n$).
\end{claim}
\begin{proof}
%First we will prove that the number of bad $3$-walks is $O(n^2)$. For any vertex $x\in V(H)$ if a path $yxz$ is a bad $2$-path then $zy$ is an edge of $G$ so the number of bad $2$-paths with $x$ in the middle is at most $2\cdot8\abs{N(x)}=16d_G(x)$. Number of walks which are not paths whose middle vertex is $x$ is exactly $d_G(x)$. So the total number of (ordered) bad $2$-walks is $\sum_{x\in V(H)}17d_G(x)\leq 2720 \sqrt{n} \times n = 2720 n^{3/2}$ (here we used Claim \ref{maxdegree}).

%For a bad $2$-walk $xyz$ the number of $3$-walks containing it is at most $d_G(x)+d_G(z) < 320\sqrt n$ by Claim \ref{maxdegree}. So the total number of bad $3$-walks is at most $2720 n^{3/2}\cdot 320\sqrt n = 870400n^2$.
First we will prove that the number of (ordered) $3$-walks that are not good $3$-paths is at most $5440n^{3/2}\overline d_G$. 

For any vertex $x\in V(H)$ if a path $yxz$ is a bad $2$-path then $zy$ is an edge of $G$, so the number of (ordered) bad $2$-paths whose middle vertex is $x$, is at most 2 times the number of edges in $G[N(x)]$, which is less than $2\cdot8\abs{N(x)}=16d_G(x)$ by Lemma \ref{neighborhood edges}. The number of $2$-walks which are not $2$-paths and whose middle vertex is $x$ is exactly $d_G(x)$. So the total number of (ordered) $2$-walks that are not good $2$-paths is at most $\sum_{x\in V(H)}17d_G(x)=17n\overline d_G$. 

Notice that, by definition, any (ordered) $3$-walk that is not a good $3$-path must contain a $2$-walk that is not a good $2$-path.
Moreover, if  $xyz$ is a $2$-walk that is not a good $2$-path, then the number of $3$-walks in $G$ containing it is at most $d_G(x)+d_G(z) < 320\sqrt n$ (for large enough $n$) by Claim \ref{maxdegree}. Therefore, the total number of (ordered) $3$-walks that are not good $3$-paths is at most $17n\overline d_G\cdot 320\sqrt n = 5440n^{3/2}\overline d_G$.

By the Blakley-Roy inequality, the total number of (ordered) 3-walks in $G$ is at least $n\overline d^3_G$. By the above discussion, all but at most $5440n^{3/2}\overline d_G$ of them are good $3$-paths, so letting $C_0 = 5440$ completes the proof of the claim.
%As we proved above number of $3$-paths which are not paths or are bad $3$-paths is $5440n^{3/2}\overline d_G$, combining them completes the proof of the claim.
%Let $abc\in H$. Let $A,B$ and $C$ be the set of good $2$-paths starting from $a,b,$ and $c$ respectively.
\end{proof}

%%%%%%%%%%%%%%%%%%%%%%%%%%%%%%%%%%%%%%%%%%%%%%%%%%%
 
\begin{claim}
\label{trianglecomponent}
Let $\{a,b,c\}$ be the vertex set of a triangle that belongs to $\D$. (By Observation \ref{2pathfatedge} (a) $abc\in H$.) Then the number of good $3$-paths whose first edge is $ab, bc$ or $ca$ is at most $8n+C_1\sqrt n$ for some constant $C_1$ and for large enough $n$.
\end{claim}
\begin{proof}
Let $S_{abc}=N(a)\cap N(b)\cap N(c)$.
For each $\{x,y\}\subset \{a,b,c\}$, let $S_{xy}=N(x)\cap N(y)\setminus \{a,b,c\}$. For each $x \in \{a,b,c\}$, let $S_x=N(x)\setminus (N(y) \cup N(z) \cup \{a,b,c\})$ where $\{y, z\}= \{a, b, c\} \setminus \{x\}$.

For each $x\in \{a,b,c\}$, let $\P_x$ be the set of good $2$-paths $xuv$ where $u\in S_x$. Let $S'_x=\{v\mid xuv\in \P_x\}$.
For each $\{x,y\}\subset \{a,b,c\}$, let $\P_{xy}$ be the set of good $2$-paths $xuv$ and $yuv$ where $u\in S_{xy}$. Let $S'_{xy} =\{v\mid xuv\in \P_{xy}\}$.

Let $\{x,y\}\subset \{a,b,c\}$ and $z=\{a,b,c\}\setminus \{x,y\}$. 
%Notice that if $xyuv$ is a good $3$-path then $u\notin N(x)$, so $u\notin S_{abc}\cup S_{xy}$. 
Notice that each $2$-path $yuv \in \P_{xy}$ ($xuv \in \P_{xy}$), is contained in exactly one good $3$-path $zyuv$ (respectively $zxuv$) whose first edge is in the triangle $abc$. Indeed, since $u\in S_{xy}$, $xyuv$ (respectively $yxuv$) is not a good $3$-path. Therefore, the number of good $3$-paths whose first edge is in the triangle $abc$, and whose third vertex is in $S_{xy}$ is $\abs{\P_{xy}}$. The number of paths in $\P_{xy}$ that start with the vertex $x$ is less than $2\abs{S'_{xy}}+ 7680 \sqrt{n}$, by Lemma \ref{weirdcor}. Similarly, the number of paths in $\P_{xy}$ that start with the vertex $y$ is less than $2\abs{S'_{xy}}+ 7680 \sqrt{n}$. Since every path in $\P_{xy}$ starts with either $x$ or $y$, we have $\abs{\P_{xy}} < 4\abs{S'_{xy}}+15360 \sqrt{n}$. Therefore, for any $\{x,y\}\subset \{a,b,c\}$, the number of good $3$-paths whose first edge is in the triangle $abc$, and whose third vertex is in $S_{xy}$ is less than $4\abs{S'_{xy}}+15360 \sqrt{n}$.

In total, the number of good $3$-paths whose first edge is in the triangle $abc$ and whose third vertex is in $S_{ab}\cup S_{bc} \cup S_{ac}$ is at most 
\begin{equation}
\label{3}
    4(\abs{S'_{ab}}+\abs{S'_{bc}}+\abs{S'_{ac}})+46080 \sqrt{n}.
\end{equation}

Let $x \in \{a,b,c\}$ and $\{y, z\} = \{a,b,c\} \setminus \{x\}$.  For any $2$-path $xuv\in \P_x$ there are $2$ good $3$-paths with the first edge in the triangle $abc$, namely $yxuv$ and $zxuv$. So the total number of $3$-paths whose first edge is in the triangle $abc$ and whose third vertex is in $S_a\cup S_b\cup S_c$ is $2(\abs{\P_a}+\abs{\P_b}+\abs{\P_c})$, which is at most
\begin{equation}
\label {4}
4(\abs{S'_{a}}+\abs{S'_{b}}+\abs{S'_{c}})+46080  \sqrt{n},
\end{equation}
by Lemma \ref{weirdcor}.

Now we will prove that every vertex is in at most $2$ of the sets $S'_a,S'_b,S'_c,S'_{ab},S'_{bc},S'_{ac}$.
Let us assume by contradiction that a vertex $v\in V(G)\setminus \{a,b,c\}$ is in at least $3$ of them. 
We claim that there do not exist $3$ vertices $u_a\in N(a)\setminus \{b,c\}$, $u_b\in N(b)\setminus \{a,c\}$ and $u_c\in N(c)\setminus \{a,b\}$ such that $xu_xv$ is a good $3$-path for each $x\in \{a,b,c\}$. Indeed, otherwise, consider hyperedges $h_a, h'_a$ containing the pairs $au_a$ and $u_av$ respectively (since $au_av$ is a good $2$-path, note that $h_a \not = h_a'$), and hyperedges $h_b, h'_b, h_c, h'_c$ containing the pairs $bu_b, u_bv, cu_c, u_cv$ respectively. Then either $h_a'\not =h_b'$ or $h_a'\not = h_c'$, say $h_a'\not =h_b'$ without loss of generality. Then the hyperedges $h_a,h_a',h_b',h_b, abc$ create a Berge $5$-cycle in $H$, a contradiction, proving that it is impossible to have 3 vertices $u_a\in N(a)\setminus \{b,c\}$, $u_b\in N(b)\setminus \{a,c\}$ and $u_c\in N(c)\setminus \{a,b\}$ with the above mentioned property. Without loss of generality let us assume that there is no vertex $u_a\in N(a)\setminus \{b,c\}$ such that $au_av$ is a good $2$-path -- in other words, $v\notin S'_a\cup S'_{ab}\cup S'_{ac}$. However, since we assumed that $v$ is contained in at least 3 of the sets $S'_a,S'_b,S'_c,S'_{ab},S'_{bc},S'_{ac}$, we can conclude that $v$ is contained in all 3 of the sets $S'_b$, $S'_c$, $ S'_{bc}$, i.e., there are vertices $u_b\in S_b, u_c\in S_c, u\in S_{bc}$ such that $vu_bb,vu_cc,vub,vuc$ are good $2$-paths. Using a similar argument as before, if $vu\in h$, $vu_b\in h_b$ and $vu_c\in h_c$, without loss of generality we can assume that $h\not = h_b$, so the hyperedges $abc$,$h$,$h_b$ together with hyperedges containing $uc$ and $u_bb$ form a Berge $5$-cycle in $H$, a contradiction.

So we proved that $$2\abs{S'_a\cup S'_b \cup S'_c\cup S'_{ab} \cup S'_{bc}\cup S'_{ac}}\geq \abs{S'_{a}}+\abs{S'_{b}}+\abs{S'_{c}}+\abs{S'_{ab}}+\abs{S'_{bc}}+\abs{S'_{ac}}$$
This together with \eqref{3} and \eqref{4}, we get that the number of good $3$-paths whose first edge is in the triangle $abc$ is at most
$$8\abs{S'_a\cup S'_b \cup S'_c\cup S'_{ab} \cup S'_{bc}\cup S'_{ac}}+ 92160\sqrt{n}<8n+C_1\sqrt{n}$$
for $C_1=92160$ and large enough $n$, finishing the proof of the claim.
\end{proof}

\begin{claim}
\label{2pathcomponent}
Let $P=abc$ be a $2$-path and $P\in \D$. (By Observation \ref{2pathfatedge} (b) $abc\in H$.) Then the number of good $3$-paths whose first edge is $ab$ or $bc$ is at most $4n+C_2\sqrt n$ for some constant $C_2>0$ and large enough $n$. 
\end{claim}
\begin{proof}
First we bound the number of $3$-paths whose first edge is $ab$. Let $S_{ab}=N(a)\cap N(b)$.
Let $S_a=N(a) \setminus (N(b) \cup \{b\})$ and $S_b=N(b)\setminus (N(a) \cup \{a\})$.  
For each $x\in \{a,b\}$, let $\P_x$ be the set of good $2$-paths $xuv$ where $u\in S_x$, and let $S_x'=\{v\mid xuv\in \P_x\}$.
The set of good $3$-paths whose first edge is $ab$ is $\P_a\cup \P_b$, because the third vertex of a good $3$-path starting with an edge $ab$ can not belong to $N(a)\cap N(b)$ by the definition of a good $3$-path.

We claim that $\abs{S'_a\cap S'_b}\leq 160\sqrt{n}$. Let us assume by contradiction that $v_0,v_1,\ldots v_k\in S'_a\cap S'_b$ for $k > 160\sqrt{n}$. For each vertex $v_i$ where $0\leq i \leq k$, there are vertices $a_i\in S_a$ and $b_i\in S_b$ such that $aa_iv_i,bb_iv_i$ are good $2$-paths. For each $0\leq i \leq k$, the hyperedge $a_iv_ib_i$ is in $H$, otherwise we can find distinct hyperedges containing the pairs $aa_i,a_iv_i,v_ib_i,b_ib$ and these hyperedges together with $abc$, would form a Berge $5$-cycle in $H$, a contradiction.
We claim that there are $j,l\in \{0,1,\ldots, k\}$ such that $a_j\not = a_l$, otherwise there is a vertex $x$ such that $x=a_i$ for each $0\le i \le k$. Then $xv_i\in G$ for each $0\le i \le k$, so we get that $d_G(x)>k > 160\sqrt{n}$ which contradicts Claim \ref{maxdegree}.

So there are $j,l\in \{0,1,\ldots, k\}$ such that $a_j\not = a_l$ and $a_jv_jb_j,a_lv_lb_l\in H$. By observation \ref{2pathfatedge} (b), there is a hyperedge $h\not = abc$ such that $ac\subset h$. Clearly either $a_j\notin h$ or $a_l\notin h$. Without loss of generality let $a_j\notin h$, so there is a hyperedge $h_a$ with $aa_j \subset h_a \not = h$. Let $h_b\supset b_jb$, then the hyperedges $abc,h,h_a,a_jv_jb_j,h_b$ form a Berge $5$-cycle, a contradiction, proving that $\abs{S'_a\cap S'_b}\leq 160\sqrt{n}$.

%We claim that $\abs{S'(a)\cap S'(b)}<2$. Let us assume by contradiction that $v,v'\in S'(a)\cap S'(b)$, i.e., there are vertices $u_a,u_a'\in S_a$ and $u_b,u_b'\in S_b$ such that $au_av,au_a'v,bu_bv'$ and $bu_b'v'$ are good $2$-paths. The hyperedge $u_avu_b$ is in $H$, otherwise hyperedges containing $au_a,u_av,vu_b,ub_b$ are all distinct and they together with $abc$, they create a Berge $5$-cycle. Similarly we can prove that $u_a'v'u_b'\in H$. By observation \ref{2pathfatedge} there is a hyperedge $h\not = abc$ such that $ac\subset h$. Clearly either $u_a\notin h$ or $u_a'\notin h$, without loss of generality let $u_a\notin h$, so there is a hyperedge $h_a$ with $au_a\subset h_a\not = h$. Let $h_b\supset u_bb$, then the hyperedges $abc,h,h_a,u_avu_b,h_b$ form a Berge $5$-cycle.
Notice that $\abs{S'_a}+\abs{S'_b} = \abs{S'_a\cup S'_b} + \abs{S'_a\cap S'_b}\leq n + 160 \sqrt{n}$. So by Lemma \ref{weirdcor}, we have $$\abs{\P_a}+\abs{\P_b}\leq 2(\abs{S'_a}+\abs{S'_b})+2\cdot7680\sqrt{n}\leq 2(n + 160 \sqrt{n}) + 2\cdot7680\sqrt{n}=2n+15680 \sqrt n$$
for large enough $n$.
So the number of good $3$-paths whose first edge is $ab$ is at most $2n+15680\sqrt n$. By the same argument, the number of good $3$-paths whose first edge is $bc$ is at most $2n+15680\sqrt n$. Their sum is at most $4n+C_2\sqrt {n}$ for $C_2=31360$ and large enough $n$, as desired.
\end{proof}

\begin{claim}
\label{k4pathcomponent} Let $\{a,b,c,d\}$ be the vertex set of a $K_4$ that belongs to $\D$.
%on vertices $a,b,c$ and $d$. 
Let $F = K_4^3$ be a hypergraph on the vertex set $\{a,b,c,d\}$.  (By Observation \ref{2pathfatedge} (c) $F \subseteq H$.)
%Clearly $D=\partial F\in \D$. 
Then the number of good $3$-paths whose first edge belongs to $\partial F$ is at most $6n+C_3\sqrt n$ for some constant $C_3>0$ and large enough $n$. 
\end{claim}
\begin{proof}
First, let us observe that there is no Berge path of length $2,3$ or $4$ between distinct vertices $x,y\in \{a,b,c,d\}$ in the hypergraph $H\setminus F$, because otherwise this Berge path together with some edges of $F$ will form a Berge $5$-cycle in $H$. This implies, that there is no path of length $3$ or $4$ between $x$ and $y$ in $G\setminus \partial F$, because otherwise we would find a Berge path of length $2,3$ or $4$ between $x$ and $y$ in $H\setminus F$.

Let $S=\{u\in V(H)\setminus \{a,b,c,d\}\mid \exists \{x,y\}\subset \{a,b,c,d\}, u\in N(x)\cap N(y)\}$.
For each $x\in \{a,b,c,d\}$, let $S_x=N(x)\setminus (S\cup \{a,b,c,d\})$. Let $\P_S$ be the set of good $2$-paths $xuv$ where $x\in \{a,b,c,d\}$ and $u\in S$. Let $S'=\{v\mid xuv\in \P_S\}$.
For each $x\in \{a,b,c,d\}$, let $\P_x$ be the set of good $2$-paths $xuv$ where $u\in S_x$, and let $S_x'=\{v\mid xuv\in \P_x\}$.

Let $v\in S'$. By definition, there exists a pair of vertices $\{x,y\}\subset \{a,b,c,d\}$ and a vertex $u$, such that $xuv$ and $yuv$ are good $2$-paths. 

Suppose that $zu'v$ is a $2$-path different from $xuv$ and $yuv$ where $z \in \{a,b,c,d\}$. If $u'=u$ then $z\notin \{x,y\}$ so there is a Berge $2$-path between $x$ and $y$ or between $x$ and $z$ in $H \setminus F$, which is impossible. So $u\not=u'$. Either $z\not=x$ or $z\not=y$, without loss of generality let us assume that $z\not=x$. Then $zu'vux$ is a path of length $4$ in $G\setminus \partial F$, a contradiction.
 So for any $v\in S'$ there are only $2$ paths of $\P_a\cup \P_b \cup \P_c \cup \P_d \cup \P_S$ that contain $v$ as an end vertex -- both of which are in $\P_S$ -- which means that $v\notin S_a' \cup S_b'\cup S_c' \cup S_d'$, so $S'\cap (S_a'\cup S_b'\cup S_c' \cup S_d') = \emptyset$. Moreover,
 \begin{equation}
\label{boundingP_S}
\abs{\P_S}\leq 2\abs{S'}.
 \end{equation}

We claim that $S_a'$ and $S_b'$ are disjoint. Indeed, otherwise, if $v\in S_a'\cap S_b'$ there exists $x\in S_a$ and $y\in S_b$ such that $vxa$ and $vyb$ are paths in $G$, so there is a $4$-path $axvyb$ between vertices of $F$ in $G\setminus \partial F$, a contradiction. Similarly we can prove that $S_a', S_b', S_c'$ and $S_d'$ are pairwise disjoint. This shows that the sets $S', S_a', S_b', S_c'$ and $S_d'$ are pairwise disjoint. So we have 
\begin{equation}
\label{unionofS_x}
\abs{S'\cup S_a'\cup S_b'\cup S_c' \cup S_d'}=\abs{S'}+\abs{S_a'}+\abs{S_b'}+\abs{S_c'}+\abs{S_d'}.
\end{equation}

By Lemma \ref{weirdcor}, we have $\abs{\P_a}+\abs{\P_b}+\abs{\P_c}+\abs{\P_d}\leq2(\abs{S_a'}+\abs{S_b'}+\abs{S_c'}+\abs{S_d'})+4\cdot 7680\sqrt{n}.$
Combining this inequality with \eqref{boundingP_S}, we get
\begin{equation}
\label{longinequality}
    \abs{\P_S}+\abs{\P_a}+\abs{\P_b}+\abs{\P_c}+\abs{\P_d}\leq2 \abs{S'} + 2(\abs{S_a'}+\abs{S_b'}+\abs{S_c'}+\abs{S_d'})+4\cdot 7680\sqrt{n}.
\end{equation}

Combining \eqref{unionofS_x} with \eqref{longinequality} we get
\begin{equation}
\label{combine6and7}
\abs{\P_S}+\abs{\P_a}+\abs{\P_b}+\abs{\P_c}+\abs{\P_d}\leq 2\abs{S' \cup S_a'\cup S_b'\cup S_c' \cup S_d'}+30720\sqrt n<2n+30720\sqrt n,
\end{equation}
for large enough $n$.
%since $n>\abs{S'\cup S_a'\cup S_b'\cup S_c' \cup S_d'}$ and $d_G(a)+d_G(b)+d_G(c)+d_G(d)<4\cdot 60d=240d$ we get that the number of good $2$-paths starting from 
%$$\P_a\cup \P_b \cup \P_c \cup \P_d \cup \P_S\leq2n+24\cdot 240d$$

Each $2$-path in $\P_S \cup \P_a\cup \P_b \cup \P_c \cup \P_d$ can be extended to at most three good $3$-paths whose first edge is in $\partial F$.
(For example, $auv\in \P_a$ can be extended to $bauv,cauv$ and $dauv$.) On the other hand, every good $3$-path whose first edge is in $\partial F$ must contain a $2$-path of $\P_a\cup \P_b \cup \P_c \cup \P_d \cup \P_S$ as a subpath. So the number of good $3$-paths whose first edge is in $\partial F$ is at most  $3 \abs{\P_a\cup \P_b \cup \P_c \cup \P_d \cup \P_S} = 3 (\abs{\P_S}+\abs{\P_a}+\abs{\P_b}+\abs{\P_c}+\abs{\P_d})$ which is at most $6n+C_3\sqrt n$ by \eqref{combine6and7}, for $C_3 = 92160$ and large enough $n$, proving the desired claim.
\end{proof}

%For $x \in \{a,b,c\}$, let $H_x$ be a subhypergraph of $H_{abc}$ whose edge set is $E(H_x) = E^x_1 \bigcup E^x_2$ where $E^x_1 = \{h \in E(H_{abc}) \mid x \in h \text{ and } h \not = abc\}$ and $E^x_2 = \{h \in E(H_{abc}) \mid \exists h' \in E^x_1,  x \not \in h \text{ and } h \cap h' \not = \emptyset  \}$ and its vertex set is $V(H_x) = \{v \in V(H_{abc})   \mid \exists h \in E(H_x) \text{ and } v \in h \}$. Note that $\abs{E^x_1} = d^{H_x}(x) = d^H(x)-1$ and every hyperedge in $E^x_1$ contains exactly two vertices of $N^{H_x}_1(x)$ and every hyperedge in $E^x_2$ contains one vertex of $N^{H_x}_1(x)$ and two vertices of $N^{H_x}_2(x)$ because hyperedges containing more than one vertex of $N^{H_x}_1(x)$ do not belong to $H_{abc}$ (since they are in $H'_{abc}$ by property 1 of Definition \ref{garbage}) and thus, do not belong to $H_x$.

%%%%%%%%%%%%%%%%%%%%%%%%%%%%%%%%%%%%%%%%%%%%%%%%%%%%%%%%%%%%%%%%%%%%%%%
\subsection{Combining bounds on the number of $3$-paths}

%Suppose the number of edges of $G$ is $\abs{G}$, and 
Recall that $\alpha_1 \abs{G}$, $\alpha_2 \abs{G}$, $(1-\alpha_1-\alpha_2) \abs{G}$ are the number of edges of $G$ that are contained in triangles, $2$-paths and $K_4$'s of the edge-decomposition $\mathcal D$ of $G$, respectively. Then the number of triangles, $2$-paths and $K_4$'s in $\mathcal D$ is $\alpha_1\abs{G}/3$, $\alpha_2\abs{G}/2$ and $(1-\alpha_1-\alpha_2)\abs{G}/6$ respectively. Therefore, using Claim \ref{trianglecomponent}, Claim \ref{2pathcomponent} and Claim \ref{k4pathcomponent}, the total number of (ordered) good $3$-paths in $G$ is at most $$\frac{\alpha_1}{3}\abs{G} ({8n+C_1\sqrt{n}}) + \frac{\alpha_2}{2}\abs{G}(4n+C_2\sqrt{n}) + \frac{(1-\alpha_1-\alpha_2)}{6}\abs{G} (6n+C_3\sqrt{n}) \le $$$$\le \abs{G}n\left(\frac{8\alpha_1}{3} + 2\alpha_2+{(1-\alpha_1-\alpha_2)}\right)+(C_1+C_2+C_3)\sqrt{n}\abs{G}=$$
$$=\frac{n^{2}\overline d_G}{2}\left( \frac{5\alpha_1+3\alpha_2+3}{3}\right)+(C_1+C_2+C_3)\frac{n^{3/2}\overline d_G}{2}.$$

%\frac{2\alpha_1C_1+3\alpha_2C_2+6C_3-6\alpha_1C_3-6\alpha_2C_3}{6}\abs{G}
Combining this with the fact that the number of good $3$-paths is at least $n\overline d_G^3 - C_0n^{3/2}\overline d_G$ (see Claim \ref{countingpaths}), we get 
$$n\overline d_G^3  - C_0n^{3/2}\overline d_G \le \frac{n^2\overline d_G}2 \left( \frac{5\alpha_1+3\alpha_2+3}{3} \right)+(C_1+C_2+C_3)\frac{n^{3/2}\overline d_G}{2}.$$
Rearranging and dividing by $n\overline d_G$ on both sides, we get
$$\overline d_G^2\le \left( \frac{5\alpha_1+3\alpha_2+3}{6} \right)n+\frac{(C_1+C_2+C_3)}{2}\sqrt{n}+C_0\sqrt n.$$
Using the fact that $(5\alpha_1+3\alpha_2+3)/6\geq 1/2$, it follows that
$$\overline d_G^2\le \left( \frac{5\alpha_1+3\alpha_2+3}{6} \right)n\left(1+\frac{(C_1+C_2+C_3)+2C_0}{\sqrt{n}}\right ).$$

So letting $C_4=(C_1+C_2+C_3)+2C_0$ we have,
\begin{equation}
\label{degreebound}
  \overline d_G\le \sqrt{1+\frac{C_4}{\sqrt n}}\sqrt{\frac{5\alpha_1+3\alpha_2+3}{6}}\sqrt{n} < \left(1+\frac{C_4}{2\sqrt n}\right) \sqrt{\frac{5\alpha_1+3\alpha_2+3}{6}}\sqrt{n}, 
\end{equation}
for large enough $n$. By Claim \ref{alpha12}, we have  $$\abs{H}\leq \frac{\alpha_1 }{3}\abs{G}+\frac{\alpha_2}2\abs{G}+\frac{2(1-\alpha_1-\alpha_2)}{3}\abs{G} = \frac{4-2\alpha_1-\alpha_2 }{6}\frac{n \overline d_G}{2}.$$ 
Combining this with \eqref{degreebound} we get
$$\abs{H}\leq \left(1+\frac{C_4}{2\sqrt n}\right)\frac{(4-2\alpha_1-\alpha_2) }{12}\sqrt{\frac{5\alpha_1+3\alpha_2+3}{6}}n^{3/2},$$
for sufficiently large $n$. So we have
$$\ex_3(n,C_5)\leq (1+o(1))\frac{(4-2\alpha_1-\alpha_2) }{12}\sqrt{\frac{5\alpha_1+3\alpha_2+3}{6}}n^{3/2}.$$
The right hand side is maximized when $\alpha_1=0$ and $\alpha_2=2/3$, so we have
$$\ex_3(n,C_5)\leq (1+o(1))\frac{4-2/3 }{12}\sqrt{\frac{5}{6}}n^{1.5}<(1+o(1))0.2536 n^{3/2}.$$
This finishes the proof.
%%%%%%%%%%%%%%%%%%%%%%%%%%%%%%%%%%%%%%%%%%%%%%%%%%%%%%%%%%%%%%%%%%%%%%%

\section*{Acknowledgements}
The research of the authors is partially supported by the National Research, Development and Innovation Office -- NKFIH, grant K116769.

\end{document}